\documentclass[10pt]{amsart}

\usepackage{amsfonts,amsmath,amsthm,amscd,amssymb,latexsym}
\usepackage[english]{babel}

\theoremstyle{plain}

\newtheorem{thm}{Theorem}
\newtheorem{lem}[thm]{Lemma}
\newtheorem{prop}[thm]{Proposition}

\newtheorem{quest}{Question}
\newtheorem{rem}{Remark}

\newtheorem{thmm}{Theorem}

\numberwithin{equation}{section}

\begin{document}

\title{On composition of Baire functions}
\author{Olena Karlova}
\author{Volodymyr Mykhaylyuk}

\address{Department of Mathematical Analysis, Faculty of Mathematics and Informatics, Yurii Fedkovych Chernivtsi National University, Kotsyubyns'koho str., 2, Chernivtsi, Ukraine}

\begin{abstract}
 We study the maps between topological spaces whose composition with Baire class $\alpha$ maps also belongs to the  $\alpha$'th  Baire class  and give characterizations of such maps.
\end{abstract}



\maketitle

\section{Introduction}

In this paper we study right and left Baire compositors and show that under some restrictions on the domain and the range the property of being the right Baire-one compositor is equivalent to many other function properties such as piecewise continuity, $G_\delta$-measurability, ${\rm B}_1$-stability, while left compositors are exactly continuous maps.

By definition, for an ordinal $\alpha\in[0,\omega_1)$ a map $f:X\to Y$ between topological spaces is {\it the right (left) ${\rm B}_{\alpha}$-compositor for a class $\mathcal C$ of topological spaces} if for any topological space $Z\in\mathcal C$  and a map $g:Y\to Z$ (respectively, $g:Z\to X$) of the $\alpha$'th Baire class the composition $g\circ f:X\to Z$ (respectively, $f\circ g:Z\to Y$) also belongs to the $\alpha$'th Baire class. Such maps for $X=Y=\mathbb R$, $\mathcal C=\{\mathbb R\}$ and $\alpha=1$ were introduced and studied by Dongsheng Zhao in \cite{DZ}, where he proved the following result (here and throughout the paper $\mathbb R^+=(0,+\infty)$).

\begin{thmm}\label{th:B1comp}
For a function $f:\mathbb R\to\mathbb R$ the following conditions are equivalent:
  \begin{enumerate}
    \item $f$ is $G_\delta$-measurable (i.e., the set $f^{-1}(V)$ is $G_\delta$ for any open set $V\subseteq \mathbb R$);

    \item $f$ is the right ${\rm B}_1$-compositor;

    \item for any Baire-one function $\varepsilon:\mathbb R\to \mathbb R^+$ there exists a function $\delta:\mathbb R\to \mathbb R^+$ such that
        \begin{gather}\label{gath:intro}
        |x-y|<\min\{\delta(x),\delta(y)\}\,\,\Longrightarrow \,\, |f(x)-f(y)|<\min\{\varepsilon(f(x)),\varepsilon(f(y))\}.
        \end{gather}
  \end{enumerate}
\end{thmm}

Moreover, it was introduced the notion of a $k$-continuous function in~\cite{DZ}. Namely, a function $f:\mathbb R\to \mathbb R$ is called {\it $k$-continuous} if for any function $\varepsilon:\mathbb R\to \mathbb R^+$ there exists a function $\delta:\mathbb R\to \mathbb R^+$ which satisfies~(\ref{gath:intro}) for all  $x,y\in\mathbb R$. Obviously, Theorem~\ref{th:B1comp} implies that every $k$-continuous function is the right ${\rm B}_1$-compositor. In this connection D.~Zhao  posed the following question \cite[p.~548]{DZ}.

\begin{quest}\label{que:1}
Is every right  ${\rm B}_1$-compositor $f:\mathbb R\to\mathbb R$  a $k$-continuous function?
\end{quest}

The positive answer to this question was given independently in~\cite{FC:right} and \cite{KS:MatStud:2012}.
Observe that crucial auxiliary results were a Jayne--Rogers theorem~\cite[Theorem 1]{JayneRogers} in~\cite{FC:right} and a Banakh--Bokalo theorem~\cite[Theorem 8.1]{BaBo}  in \cite{KS:MatStud:2012} on the equivalence of the $G_\delta$-measurability of a function to the piecewise continuity.

Our goal is to obtain a characterization of the right Baire compositors between arbitrary metrizable (non-separable) spaces and the following theorem (see Theorem~\ref{tururu}) is our first main result.
\begin{thmm}\label{intro:tururu}
   Let $(X,d_X)$ be a metric space, $(Y,d_Y)$ be a metric space and $f:X\to Y$  be a map. Consider the following conditions:
    \begin{enumerate}
       \item\label{ittt:1} $f$ is of the first stable Baire class (i.e., there exists a sequence of continuous maps $f_n:X\to Y$ such that for every $x\in X$ there is $k\in\mathbb N$ with $f_n(x)=f(x)$ for all $k\ge n$);

      \item\label{ittt:2} $f$ is piecewise continuous (i.e., there exists a countable  closed cover $\mathcal F$ of $X$ such that $f|_F$ is continuous for every $F\in\mathcal F$);

      \item\label{ittt:3} for any function $\varepsilon:Y\to \mathbb R^+$ there exists a function $\delta:X\to \mathbb R^+$ such that for all  $x,y\in X$
    \begin{gather}\label{gath:intro2}
      d_X(x,y)<\min\{\delta(x),\delta(y)\}\,\,\Longrightarrow \,\, d_Y(f(x),f(y))<\min\{\varepsilon(f(x)),\varepsilon(f(y))\};
    \end{gather}

      \item\label{ittt:4} for any function  $\varepsilon:Y\to \mathbb R^+$ of the first Baire class  there exists a function $\delta:X\to\mathbb R^+$ such that (\ref{gath:intro2}) holds  for all $x,y\in X$;

     \item\label{ittt:5}  $f$ is the right ${\rm B}_1$-compositor for the class of all metrizable connected and locally path-connected spaces;

      \item\label{ittt:6} $f$ is $G_\delta$-measurable and $\sigma$-discrete (which means that there exists a family $\mathcal B=\bigcup\limits_{n=1}^\infty\mathcal B_n$ of subsets of $X$, which is called {\it a base for $f$}, such that for any open set $V\subseteq Y$ there is a subfamily $\mathcal B_V\subseteq\mathcal B$ with $f^{-1}(V)=\bigcup\mathcal B_V$  and each family $\mathcal B_n$ is discrete in $X$).
    \end{enumerate}
Then (\ref{ittt:1})$\Rightarrow$(\ref{ittt:2})$\Rightarrow$(\ref{ittt:3})$\Rightarrow$(\ref{ittt:4})$\Rightarrow$(\ref{ittt:5})$\Leftrightarrow$(\ref{ittt:6}).
        If  $X$ is a hereditarily Baire space, then (\ref{ittt:6})$\Rightarrow$(\ref{ittt:2}).
If, moreover, $Y$ is a path-connected space and $Y\in\sigma {\rm AE}(X)$, then all the conditions (\ref{ittt:1})--(\ref{ittt:6}) are equivalent.
\end{thmm}

We also prove the following result (see Theorem~\ref{th:general_sigma_H_barely}) which serves as an important tool in the proof of Theorem~\ref{intro:tururu}.

\begin{thmm}\label{th:intro:sigma}
 Let $(X,d_X)$ be a hereditarily Baire metric space and $(Y,d_Y)$ be a metric space. For a map $f:X\to Y$  the following conditions are equivalent:
  \begin{enumerate}
    \item\label{it:intro:1} $f$ is a $\sigma$-discrete map with a base which consists of $F_\sigma$-subsets of $X$;

    \item\label{it:intro:2} $f$ belongs to the first Lebesgue class (i.e., the set $f^{-1}(V)$ is $F_\sigma$ in $X$ for any open set $V\subseteq Y$);

    \item\label{it:intro:3} $f$ is barely continuous (i.e., for each non-empty closed subset $F\subseteq X$ the restriction $f|_F$ has a continuity point);

    \item\label{it:intro:4} for any $\varepsilon>0$ there exists a Baire-one function $\delta:X\to \mathbb R^+$ such that for all $x,y\in X$
  \begin{gather}\label{gath:intro3}
    d_X(x,y)<\min\{\delta(x),\delta(y)\} \,\,\Longrightarrow\,\, d_Y(f(x),f(y))<\varepsilon.
 \end{gather}

 \item\label{it:intro:5} for any $\varepsilon>0$ there exists a function $\delta:X\to \mathbb R^+$ such that
  (\ref{gath:intro3}) holds  for all $x,y\in X$.
  \end{enumerate}
  \end{thmm}

  Notice that the equivalence of conditions (\ref{it:intro:1}) and (\ref{it:intro:2}) in Theorem~\ref{th:intro:sigma} was established in \cite{Hansell:1971} in the case of absolutely analytic (in particular, complete) metric space $X$. The equivalence of conditions (\ref{it:intro:2}) and (\ref{it:intro:5}) was firstly obtained in~\cite{LTZ}  for Polish spaces $X$ and $Y$ (see also~\cite{FC:new} for an alternative proof). Condition (\ref{it:intro:5}) for functions $f:\mathbb R\to\mathbb R$ was also studied in~\cite{JLL,Lec,Sark}.

 Finally, the next our theorem (see Theorem~\ref{tarara}) gives a characterization of left Baire compositors.

 \begin{thmm}\label{th:intro:left}
   Let $X$ be a $T_1$-space, $Y$ be a perfectly normal space, $f:X\to Y$ be a map and $\alpha\in[1,\omega_1)$.
If one of the following conditions holds:
\begin{itemize}
  \item[(i)] $\alpha=1$ and $X$ is a connected and locally path-connected metrizable space, or

  \item[(ii)] $\alpha>1$ and $X$ is a first countable space  such that  for any finite sequence $U_1,\dots, U_n$ of open subsets of $X$ there exists a continuous map $\varphi:[1,n]\to X$ with $\varphi(i)\in U_i$ for every $i\in\{1,n\}$,
\end{itemize}
then the following conditions are equivalent:
   \begin{enumerate}
      \item $f$ is continuous;

      \item $f$ is the left ${\rm B}_\alpha$-compositor.
      \end{enumerate}
 \end{thmm}

Observe that similar characterization was proved in \cite{FC:left} for $X=Y=\mathbb R$ and $\alpha=1$.

\section{Terminology and notations}

Let  $X$, $Y$ be topological spaces. By ${\rm C}(X,Y)$ we denote the collection of all continuous maps between $X$ and $Y$.

A sequence $(f_n)_{n=1}^\infty$ of maps $f_n:X\to Y$ is {\it stably convergent to a map $f:X\to Y$ on $X$} (we denote this fact by $f_n\stackrel{{\rm st}}{\mathop{\longrightarrow}} f$) if for every $x\in X$ there exists $k\in\mathbb N$ such that $f_n(x)=f(x)$ for all  $n\ge k$.
If $A\subseteq Y^X$, then the symbols $\overline{A}^{\,\,{\rm st}}$ and $\overline{A}^{\,\,{\rm p}}$ stands for the collection of all stable and pointwise  limits of sequences of maps from $A$, respectively.

We  inductively define Baire classes and stable Baire classes as follows: let
$$
{\rm B}_0(X,Y)={\rm B}_0^{{\rm st}}(X,Y)={\rm C}(X,Y)
$$
and for each ordinal $\alpha\in (0,\omega_1)$ let ${\rm B}_\alpha(X,Y)$ be the family of all maps of {\it the $\alpha$'th Baire class} and ${\rm B}_\alpha^{\rm st}(X,Y)$ be the family of all maps of {\it the $\alpha$'th stable Baire class} which defined by the rules
$$
{\rm B}_\alpha(X,Y)=\overline{\bigcup\limits_{\beta<\alpha}{\rm B}_\beta(X,Y)}^{\,\,{\rm p}}\,\,\,\,\mbox{and}\,\,\,\, {\rm B}_\alpha^{\rm st}(X,Y)=\overline{\bigcup\limits_{\beta<\alpha}{\rm B}_\beta^{\rm st}(X,Y)}^{\,\,{\rm st}}.
$$

{\it  The right ${\rm B}_{\alpha}$-compositor (for a class  $\mathcal C$ of topological spaces)} is a map $f:X\to Y$ such that for any space $Z$  (for any  $Z\in\mathcal C$) and a map $g\in {\rm B}_\alpha(Y,Z)$ the composition  $g\circ f:X\to Z$ is of the $\alpha$'th Baire class.

Similarly, a map  $f:X\to Y$ is called {\it the left ${\rm B}_\alpha$-compositor (for a class  $\mathcal C$ of topological spaces)} if for any space $Z$ (for any  $Z\in\mathcal C$) and a map $g\in {\rm B}_\alpha(Z,X)$ the composition $f\circ g:Z\to Y$ belongs to the $\alpha$'th Baire class.

Let $\mathcal M_0(X)$ be the family of all functionally closed subsets of  $X$ and let $\mathcal A_0(X)$ be the family of all functionally open subsets of  $X$. For every  $\alpha\in [1,\omega_1)$ we put
\begin{gather*}
   \mathcal M_{\alpha}(X)=\Bigl\{\bigcap\limits_{n=1}^\infty A_n: A_n\in\bigcup\limits_{\beta<\alpha}\mathcal A_{\beta}(X),\,\, n=1,2,\dots\Bigr\}\,\,\,\mbox{and}\,\,\,
   \mathcal A_{\alpha}(X)=\Bigl\{\bigcup\limits_{n=1}^\infty A_n: A_n\in\bigcup\limits_{\beta<\alpha}\mathcal M_{\beta}(X),\,\, n=1,2,\dots\Bigr\}.
 \end{gather*}
Elements from the class  $\mathcal M_\alpha(X)$ belong to {\it the $\alpha$'th functionally multiplicative class} and elements from  $\mathcal A_\alpha(X)$ belong to {\it the $\alpha$'th functionally additive class} in  $X$. We say that a set is {\it functionally ambiguous of the $\alpha$'th class} if it belongs to $\mathcal M_\alpha(X)\cap\mathcal A_\alpha(X)$.

A map $f:X\to Y$ is of the {\it $\alpha$'th (functionally) Lebesgue class}, if the preimage $f^{-1}(V)$ of any open set $V\subseteq Y$ is of the $\alpha$'th (functionally) additive class $\alpha$ in $X$. The collection of all maps of the $\alpha$'th (functionally) Lebesgue class we denote by ${\rm H}_{\alpha}(X,Y)$ (${\rm K}_{\alpha}(X,Y)$).

A family $\mathcal A=(A_i:i\in I)$ of subsets of a topological space $X$ is called
   {\it discrete}, if every point of $X$ has an open neighborhood which intersects at most one set from $\mathcal A$;
   {\it strongly discrete}, if there exists a discrete family $(U_i:i\in I)$ of open subsets of $X$ such that $\overline{A_i}\subseteq U_i$ for every $i\in I$;
   {\it strongly functionally discrete} or, briefly, {\it an sfd family}, if there exists a discrete family $(U_i:i\in I)$ of functionally open subsets of $X$ such that $\overline{A_i}\subseteq U_i$ for every $i\in I$. Notice that in a metrizable space $X$ each discrete family of sets is strongly functionally discrete.

A family $\mathcal B$ of sets of a topological space $X$ is called {\it a base}  for a map $f:X\to Y$ if the preimage $f^{-1}(V)$ of an arbitrary open set  $V$ in $Y$ is a union of sets from $\mathcal B$. In the case when $\mathcal B$ is a countable union of (strongly functionally) discrete families, we say that $f$ is $\sigma$-(strongly functionally) discrete and write this fact as $f\in\Sigma(X,Y)$ ($f\in \Sigma^f(X,Y)$). If a map $f$ has a $\sigma$-(strongly functionally) discrete base which consists of (functionally) ambiguous sets of the $\alpha$'th class, then we say that $f$ belongs to {\it the class $\Sigma_\alpha(X,Y)$ (or to $\Sigma_\alpha^f(X,Y)$, respectively)}.

We will use the next result which, in fact, was established in \cite{Karlova:TA:2015} and \cite{Karlova:EJM:2015}.

\begin{thm}\label{th:LHB_general}
  Let  $\alpha\in[0,\omega_1)$, $X$ be a topological space, $Y$ be a metrizable connected and locally path-connected space. Then
  \begin{enumerate}
  \item[(i)] ${\rm B}_{\alpha}(X,Y)={\Sigma}_{\alpha}^f(X,Y)$, if $\alpha<\omega_0$,

  \item[(ii)] ${\rm B}_{\alpha}(X,Y)={\Sigma}_{\alpha+1}^f(X,Y)$, if $\alpha\ge \omega_0$.
\end{enumerate}
\end{thm}

\begin{proof}
   For  $\alpha=0$ the statement is evident. The equality (i) for $\alpha=1$ was proved in \cite[Theorem 5]{Karlova:EJM:2015}. This fact and \cite[Theorem 22]{Karlova:TA:2015} imply the equalities  (i) and (ii) for all $\alpha>1$.
\end{proof}

A topological space  is called {\it perfect} if any its open subset is $F_\sigma$.

\section{An ''$\varepsilon-\delta$'' characterization of $\sigma$-discrete maps}\label{sec:eps-delta}

\begin{lem}\label{lemma:criterium_Leb}
 Let $\alpha\in[1,\omega_1)$, $X$ be a topological space and $(Y,d)$ be a metric space. For a map $f:X\to Y$ the following conditions are equivalent:
  \begin{enumerate}
    \item\label{it:lemma:crit_leb:1} $f\in \Sigma_\alpha^{f}(X,Y)$;

    \item\label{it:lemma:crit_leb:2} for any $\varepsilon>0$ there exists a $\sigma$-sfd family $\mathcal A_\varepsilon$ which consists of functionally ambiguous sets of the class $\alpha$  in $X$ such that $X=\bigcup\mathcal A_\varepsilon$ and ${\rm diam}f(A)<\varepsilon$ for every $A\in\mathcal A_\varepsilon$.
   \end{enumerate}
\end{lem}

\begin{proof}
{\bf (\ref{it:lemma:crit_leb:1})$\Rightarrow$(\ref{it:lemma:crit_leb:2}).} Fix $\varepsilon>0$. Let $\mathcal B$ be a  $\sigma$-sfd base for $f$ which consists of functionally ambiguous sets of the class $\alpha$  in $X$. Consider a cover $\mathcal V$ of the space $Y$ by open balls of diameters  $<\varepsilon$. Then the family
$$\mathcal A_\varepsilon=\{B\in\mathcal B: B\subseteq f^{-1}(V)\,\,\,\mbox{for some}\,\,\, V\in\mathcal V\}$$ is required.

{\bf (\ref{it:lemma:crit_leb:2})$\Rightarrow$(\ref{it:lemma:crit_leb:1}).} For every $n\in\mathbb N$ we choose a $\sigma$-sfd family $\mathcal A_n$ which consists of functionally ambiguous sets of the class $\alpha$ in $X$ such that $X=\bigcup\mathcal A_n$ and ${\rm diam}f(A)< \frac 1n$ for every $A\in\mathcal A_n$. Denote $\mathcal B=\bigcup\limits_{n} \mathcal A_{n}$. Then  $\mathcal B$ is a $\sigma$-sfd family of functionally ambiguous sets of the class  $\alpha$ and $\mathcal B$ is a base for $f$, since $f^{-1}(V)=\bigcup\{B\in\mathcal B: B\subseteq f^{-1}(V)\}$ for any open set $V\subseteq Y$.
\end{proof}

By $\tau(X)$ we denote the topology of a topological space $X$. A multi-valued map \mbox{$U:X\to \tau(X)$} is called {\it a neighborhood-map} if \mbox{$x\in U(x)$} for every $x\in X$.

 Let $\mathcal A$ and $\mathcal B$ be families of sets. We write $\mathcal A\prec\mathcal B$ if for every set $A\in\mathcal A$ there exists a set $B\in\mathcal B$ such that $A\subseteq B$.

\begin{lem}\label{lem:neighborhood_map}
  Let $X$ be a topological space and $\mathcal A$ is a $\sigma$-sfd cover of  $X$ by functionally closed maps. Then
  \begin{enumerate}
  \item\label{it:lem:neighborhood_map:1} there exists a neighborhood-map $U:X\to\tau (X)$ such that for all  $x,y\in X$
  \begin{gather}\label{gath:neighborhood_map_lem}
    x,y\in U(x)\cap U(y) \,\,\Longrightarrow\,\, x,y\in A\,\,\mbox{for some}\,\, A\in\mathcal A;
    \end{gather}
  \item\label{it:lem:neighborhood_map:2} if $X$ is metrizable and  $d_X$ is a metric on $X$ which generates its topological structure, then there exists a function $\delta\in{\rm B}_1(X,\mathbb R^+)$ such that for all $x,y\in X$
  \begin{gather}\label{gath:delta_lemma}
    d_X(x,y)<\min\{\delta(x),\delta(y)\} \,\,\Longrightarrow\,\, x,y\in A\,\,\mbox{for some}\,\, A\in\mathcal A.
  \end{gather}
  \end{enumerate}
\end{lem}

\begin{proof}
  According to \cite[Lemma 13]{Karlova:TA:2015} there exists a sequence  $(\mathcal A_n)_{n=1}^\infty$ of sfd families of functionally closed sets in  $X$ such that $\bigcup\limits_{n=1}^\infty \mathcal A_n\prec\mathcal A$,  $\bigcup\bigcup\limits_{n=1}^\infty \mathcal A_n=X$ and $\mathcal A_{n}\prec\mathcal A_{n+1}$ for every $n\in\mathbb N$.

 Let $x\in X$. Denote $n(x)=\min\{n:x\in \bigcup\mathcal A_n\}$. Since the family $\mathcal A_{n(x)}$ is disjoint, there exists a unique set $A(x)\in\mathcal A_{n(x)}$ such that $x\in A(x)$. We put
   $$
  U(x)=X\setminus (\bigcup \mathcal A_{n(x)}\setminus A(x)).
  $$
Notice that the set $U(x)$ is open, since the family $\mathcal A_{n(x)}$ is discrete.

 In the case (\ref{it:lem:neighborhood_map:2}) for every $x\in X$ we put
  $$
 \delta(x)=d_X(x,X\setminus U(x)).
  $$

  We show that the neighborhood-map  $U:X\to\tau (X)$ in the case (\ref{it:lem:neighborhood_map:1}) and the function $\delta:X\to\mathbb R^+$ in the case  (\ref{it:lem:neighborhood_map:2}) satisfy the requirements of the lemma.

   Let $x,y\in U(x)\cap U(y)$. Assume that $n(x)\le n(y)$ and choose a set $B\in A_{n(y)}$ such that $A(x)\subseteq B$. Since  $x\in U(y)$,  $x\in A(y)$. It follows that $B=A(y)$. Since $\bigcup\limits_{n=1}^\infty \mathcal A_n\prec\mathcal A$,  there exists a set  $A\in\mathcal A$ with  $A(y)\subseteq A$. Then $x,y\in A$.

Observe that in the case~(\ref{it:lem:neighborhood_map:2}) for all $x,y\in X$ the inequality $d_X(x,y)<\min\{\delta(x),\delta(y)\}$ implies that $x,y\in U(x)\cap U(y)$. Hence, we get (\ref{gath:delta_lemma}).

It remains to prove that $\delta$ is of the first Baire class. For every $n\in\mathbb N$ we denote
$$
F_n=\cup\mathcal A_n\quad\mbox{and}\quad \delta_n=\delta|_{F_n}.
$$
Since
$$
\delta_1(x)=d_X(x,F_1\setminus A)
$$
for $x\in A$, the restriction $\delta_1|_A$ is continuous for every $A\in\mathcal A_1$. Then the function $\delta_1:F_1\to\mathbb R^+$ is continuous, since the family $\mathcal A_1$ is discrete and consists of closed sets. Assume that each function $\delta_1$,\dots,$\delta_n$ belongs to the first Baire class for some $n\ge 1$. Fix $A\in\mathcal  A$ and notice that for all $x\in A$ we have
$$
\delta_{n+1}(x)=\left\{\begin{array}{ll}
                     \delta_n(x), & \mbox{if}\,\, x\in B\,\,\mbox{for some}\,\, B\in\mathcal A_n, \\
                     d_X(x,F_{n+1}\setminus A), & \mbox{otherwise}.
                   \end{array}
\right.
$$
Taking into account the inductive assumption and the closedness of $B$, we obtain that $\delta_{n+1}|_A$ is a Baire-one function. Since the family $\mathcal A_2$ is discrete,  $\delta_{n+1}\in{\rm B}_1(F_{n+1},\mathbb R^+)$.

 Since the family $(F_n:n\in\mathbb N)$ is a closed cover of  $X$ such that every restriction $\delta|_{F_n}$ is of the first Baire class,  $\delta\in{\rm B}_1(X,\mathbb R^+)$.
\end{proof}

\begin{lem}\label{lem:sigma_implies_ed}
  Let $X$ be a topological space, $(Y,d_Y)$ be a metric space and $f\in\Sigma_1^f(X,Y)$. Then for any $\varepsilon>0$ there exists
   \begin{enumerate}
   \item\label{it:lem:sigma_implies_ed:1}  a neighborhood-map $U:X\to\tau(X)$ such that for all $x,y\in X$
  \begin{equation}\label{eq:2}
  x,y\in U(x)\cap U(y)\,\,\,\Longrightarrow \,\,\, d_Y(f(x),f(y))<\varepsilon;
  \end{equation}

  \item\label{it:lem:sigma_implies_ed:2}  a function $\delta\in{\rm B}_1(X,\mathbb R^+)$ such that for all  $x,y\in X$
  \begin{gather}\label{gath:delta_lemma_tata}
    d_X(x,y)<\min\{\delta(x),\delta(y)\} \,\,\Longrightarrow\,\, d_Y(f(x),f(y))<\varepsilon,
  \end{gather}
if $X$ is metrizable and $d_X$ is a metric on $X$ which generates its topological structure.
  \end{enumerate}
\end{lem}

\begin{proof}
  Fix $\varepsilon>0$. Applying Lemma~\ref{lemma:criterium_Leb} we may choose a $\sigma$-sfd family $\mathcal A$ which consists of functionally closed sets such that $X=\bigcup\mathcal A$ and ${\rm diam}f(A)<\varepsilon$ for each $A\in\mathcal A$. Lemma~\ref{lem:neighborhood_map} implies that there exists a neighborhood-map  $U:X\to\tau (X)$ such that (\ref{gath:neighborhood_map_lem}) holds in the case (\ref{it:lem:sigma_implies_ed:1} ) and there exists a function $\delta\in{\rm B}_1(X,\mathbb R^+)$ which satisfies (\ref{gath:delta_lemma}) у випадку (\ref{it:lem:sigma_implies_ed:2}). Therefore, if  $x,y\in U(x)\cap U(y)$  in the case (\ref{it:lem:sigma_implies_ed:1}) or  $d_X(x,y)<\min\{\delta(x),\delta(y)\}$  in the case (\ref{it:lem:sigma_implies_ed:2}), then there exists a set $A\in\mathcal A$ such that $x,y\in A$. Then $d_Y(f(x),f(y))\le{\rm diam}f(A)<\varepsilon$.
\end{proof}

\begin{lem}\label{lem:ed_implies_sigma}
    Let $(X,d_X)$, $(Y,d_Y)$ be metric spaces and $f:X\to Y$ be a map. If for any $\varepsilon>0$ there exists a function $\delta:X\to \mathbb R^+$ such that for all $x,y\in X$ the inequality $d_Y(f(x),f(y))<\varepsilon$ holds whenever  $ d_X(x,y)<\min\{\delta(x),\delta(y)\}$, then  $f\in \Sigma_1^{f}(X,Y)$.
\end{lem}

\begin{proof}
 We show that the condition the condition (\ref{it:lemma:crit_leb:2}) of Lemma~\ref{lemma:criterium_Leb} holds.
Fix $\varepsilon>0$ and take a function $\delta:X\to \mathbb R^+$ such that for all  $x,y\in X$
  \begin{gather*}
    d_X(x,y)<\min\{\delta(x),\delta(y)\}\,\,\,\Rightarrow\,\,\, d_Y(f(x),f(y))<\frac{\varepsilon}{4}.
  \end{gather*}
  Now for every $n\in\mathbb N$ we choose a $\sigma$-discrete cover $\mathcal H_n$  of $X$ by closed sets of diameters $<\frac{1}{3n}$ and denote  $D_n=\{x\in X: \delta(x)>\frac 1n\}$. Let $\mathcal H=\bigcup\limits_{m}\mathcal H_{mn}$, where each family $\mathcal H_{mn}$ is discrete in $X$. We put
  $$
  \mathcal A_{mn}=\{H\cap\overline{D_n}:H\in\mathcal H_{mn}\}\quad\mbox{and}\quad \mathcal A=\bigcup\limits_{n=1}^\infty\bigcup\limits_{m=1}^\infty\mathcal A_{mn}.
  $$
 Then the family $\mathcal A$ forms a cover of $X$ and each family  $\mathcal A_{mn}$ is discrete.

Let  $m,n\in\mathbb N$ and  $x,y\in A\in\mathcal A_{mn}$. We choose a number $k>3n$ with $x,y\in D_k$. Since $x,y\in\overline{D_n}$, there exist $x_1,y_1\in   D_n$ such that $d_X(x,x_1)<\frac 1k$ and $d_X(y,y_1)<\frac 1k$. Moreover, since $x,y\in H\in\mathcal H_{mn}$, we have $d_X(x,y)<\frac{1}{3n}$, which implies $d_X(x_1,y_1)<\frac 1n$. Therefore,
\begin{gather*}
  d_Y(f(x),f(y))\le d_Y(f(x),f(x_1))+d_Y(f(x_1),f(y_1))+d_Y(f(y_1),f(y))<\frac{\varepsilon}{4}+\frac{\varepsilon}{4}+\frac{\varepsilon}{4}=\frac{3\varepsilon}{4}.
\end{gather*}
Hence, ${\rm diam}\, f(A)\le\tfrac{3\varepsilon}{4}<\varepsilon$.
\end{proof}

Let $\alpha\in[0,\omega_1)$. A map  $f:X\to Y$ is said to be {\it $\mathcal M_\alpha$-measurable}, if $f^{-1}(B)\in \mathcal M_\alpha(X)$ for any  $B\in\mathcal M_{\alpha}(Y)$.

The following fact immediately implies from the definition of $\mathcal M_\alpha$-measurable map.

\begin{prop} Let $\alpha\in[0,\omega_1)$. For a map $f:X\to Y$ the following conditions are equivalent:
\begin{enumerate}
  \item $f$ is $\mathcal M_\alpha$-measurable;

  \item $f^{-1}(B)\in \mathcal A_\alpha(X)$ for any $B\in\mathcal A_{\alpha}(Y)$;

  \item the preimage $f^{-1}(B)$ of any functionally ambiguous set $B$ of the class  ${\alpha}$ in $Y$ is a functionally ambiguous set of the class  $\alpha$ in $X$.
\end{enumerate}
\end{prop}

\begin{thm}\label{th:M1_implies_ed}
  Let $X$ be a topological space, $(Y,d_Y)$ be a metric space and $f:X\to Y$ be an $\mathcal M_1$-measurable map from the class $\Sigma^f(X,Y)$. Then for any function  $\varepsilon\in{\rm B}_1(Y,\mathbb R^+)$ there exists a neighborhood-map $U:X\to\tau(X)$ such that for all $x,y\in X$
    \begin{gather*}
      x,y\in U(x)\cap U(y)\,\,\Longrightarrow \,\, d_Y(f(x),f(y))<\min\{\varepsilon(f(x)),\varepsilon(f(y))\}.
    \end{gather*}
\end{thm}

\begin{proof}
  Fix any Baire-one function $\varepsilon:Y\to\mathbb R^+$. Notice that for every $n\in\mathbb N$ the set
  $$
  A_n=f^{-1}(\varepsilon^{-1}(\tfrac{1}{2^n},+\infty))
  $$
  belongs to the class $\mathcal A_1(X)$. By the Reduction Theorem for sets of functional classes \cite[Lemma 3.1]{Karlova:CMUC:2013} (see also \cite[p.~350]{Kuratowski:Top:1}) there exists a partition $(B_n:n\in\mathbb N)$ of $X$ by functionally ambiguous sets such that  $B_n\subseteq A_n$ for every $n$. Lemma~\ref{lem:neighborhood_map} implies that there exists a neighborhood-map $V:X\to\tau(X)$ such that for all  $x,y\in X$
  \begin{gather*}
x,y\in V(x)\cap V(y)\,\,\Longrightarrow\,\, x,y\in B_n\,\,\mbox{for some}\,\, n\in\mathbb N.
  \end{gather*}
Since $f\in\Sigma^f(X,Y)\cap {\rm K}_1(X,Y)$,  it follows from Lemma~\ref{lem:sigma_implies_ed}  that for every $n\in\mathbb N$ there exists a neighborhood-map $V_n:X\to\tau(X)$ such that for all $x,y\in X$
 \begin{gather*}
  x,y\in V_n(x)\cap V_n(y)\,\,\Longrightarrow\,\, d_Y(f(x),f(y))<\frac{1}{2^n}.
 \end{gather*}
For every $x\in X$ we put
  $$
  U(x)=V(x)\cap V_n(x),
  $$
  if $x\in B_n$ for some  $n\in\mathbb N$. It is easy to see that $U:X\to\tau(X)$ is required neighborhood-map.
\end{proof}

\section{Barely continuous $\sigma$-discrete maps}

We start this section with a generalization of the classical theorem of R.~Baire on a classification of pointwise discontinuous maps.

Recall that a map  $f:X\to Y$ is {\it barely continuous}, if the restriction  $f|_F$ on any non-empty close set $F\subseteq X$ has a point of continuity. A map $f$ is {\it pointwise discontinuous}, if the set $C(f)$  of all continuity points of $f$ is dense in $X$.

 Notice that if $X$ is a hereditarily Baire space, then $f$ is barely continuous if and only if for any non-empty closed set $F\subseteq X$ the restriction $f|_F$ is pointwise discontinuous.

 A topological space $X$ is called {\it contractible}, if there exist a point $x_0\in X$ and a continuous map $\gamma:X\times[0,1]\to X$ such that $\gamma(x,0)=x$ and $\gamma(x,1)=x_0$ for all $x\in X$.

\begin{lem}\label{lem:ext_from_cozero}
 Let $\alpha\in [1,\omega_1)$, $X$ be a topological space, $G\subseteq X$ be a functionally open set, $Y$ be a contractible space,  $y_0\in Y$ and $g\in{\rm B}_1(G,Y)$. Then the formula
   $$
    f(x)=\left\{\begin{array}{ll}
                  g(x), & \mbox{if}\,\,\, x\in G,\\
                  y_0, & \mbox{if}\,\,\, x\in X\setminus G
                \end{array}
        \right.
    $$
   defines an extension $f\in{\rm B}_1(X,Y)$ of the map $g$.
  \end{lem}

  \begin{proof} Let $\gamma:Y\times[0,1]\to Y$ be a continuous map such that $\gamma(y,0)=y$ and $\gamma(y,1)=y_0$ for all $y\in Y$.

Consider a continuous function $\varphi:X\to [0,1]$ with $G=\varphi^{-1}((0,1])$. For every  $n\in\mathbb N$ we put
    \begin{gather*}
      F_{n,1}=X\setminus G, \,\,\,F_{n,2}=\varphi^{-1}([\tfrac{1}{n+1},1]),\\
      U_{n,1}=\varphi^{-1}([0,\tfrac{1}{n+3})),\,\,\, U_{n,2}=\varphi^{-1}((\tfrac{1}{n+2},1]).
    \end{gather*}
Moreover, for all $n\in\mathbb N$  and $i\in\{1,2\}$ we choose a continuous function  $\psi_{n,i}:X\to [0,1]$ such that
    \begin{gather*}
    U_{n,i}=\psi_{n,i}^{-1}((0,1])\,\,\, \mbox{and}\,\,\, F_{n,i}=\psi_{n,i}^{-1}(1).
    \end{gather*}

Take a sequence $(g_n)_{n=1}^\infty$ of continuous maps  $g_n:G\to Y$ which is pointwise convergent to $g$ on $G$. For every  $n\in\mathbb N$ let
    $g_{n,1}(x)=y_0$ for all $x\in X$ and  $g_{n,2}(x)=g_n(x)$ for all $x\in G$. We define a map $f_n:X\to Y$ as follows:
    \begin{gather*}\label{gath:formula2}
      f_n(x)=\left\{\begin{array}{ll}
                      \gamma(g_{n,i}(x),1-\psi_{n,i}(x)),  & \mbox{if}\,\, x\in U_{n,i}\,\,\mbox{for some}\,\, i\in\{1,2\}, \\
                      y_0, & \mbox{otherwise}.
                    \end{array}
      \right.
    \end{gather*}

Fix $x\in X$. If $x\in X\setminus G$, then $f_n(x)=\gamma(g_{n,1}(x),0)=g_{n,1}(x)=y_0=f(x)$ for all $n\in\mathbb N$. If $x\in G$, then there exists  $n_0$ such that $f_n(x)=\gamma(g_{n,2},0)=g_{n,2}(x)$ for all $n\ge n_0$. Hence, $\lim\limits_{n\to\infty}f_n(x)=f(x)$.

It is easy to see that each map $f_n:X\to Y$ is continuous. Therefore, $f\in{\rm B}_1(X,Y)$.
\end{proof}

Let $X$ be a topological space, $(Y,d_Y)$  be a metric space and $\varepsilon>0$. We say that a map $f:X\to Y$  can be {\it $\varepsilon$-approximated by a Baire-one map on $X$}, if there exists a map $g\in{\rm B}_1(X,Y)$ such that $d_Y(f(x),g(x))\le\varepsilon$ for all $x\in X$;  {\it locally $\varepsilon$-approximated by Baire-one maps on $X$}, if for every  $x\in X$ there is a neighborhood $U_x$ of $x$ such that the restriction $f|_{U_x}$ can be $\varepsilon$-approximated by a Baire-one map on $U_x$.

Notice that analogs of Lemma~\ref{lem:local_uniform_approach} and Theorem~\ref{thm:barely_is_Baire1} for  $Y=\mathbb R$ were proved in \cite{Mykhaylyuk:Visnyk:2000}.

\begin{lem}\label{lem:local_uniform_approach}
 Let $X$ be a Hausdorff paracompact space, $(Y,d_Y)$ be a metric contractible locally path-connected space, $\varepsilon>0$ and let a map $f:X\to Y$ can be locally $\varepsilon$-approximated by Baire-one maps. Then $f$ can be $\varepsilon$-approximated by a Baire-one map on $X$.
 \end{lem}

\begin{proof} Since $X$ is paracompact, there exists a $\sigma$-discrete cover $\mathcal U$ of $X$ by functionally open sets such that for each  $U\in\mathcal U$ we may choose a map $f_U\in {\rm B}_1(U,Y)$ with  $d_Y(f(x),f_U)\le\varepsilon$ for all  $x\in U$. Let $\mathcal U=\bigcup\limits_{n}\mathcal U_n$, where $\mathcal U_n$ is a discrete family of functionally open sets in $X$. Denote $U_n=\cup\mathcal U_n$ for every $n\in\mathbb N$ and notice that each set  $U_n$ is functionally open. Therefore, by Lemma~\ref{lem:ext_from_cozero} for every $n\in\mathbb N$ the map $f_n\in{\rm B}_1(U_n,Y)$, defined by the equality
  $$
  f_n(x)=f_U(x), \,\,\,\mbox{if}\,\,\, x\in U\in\mathcal U_n,
  $$
can be extended to a map $g_n\in{\rm B}_1(X,Y)$.

We put
    $$
    \mathcal A_1=\mathcal U_1\,\,\,\mbox{and}\,\,\, \mathcal A_{n+1}=(U\setminus\bigcup\limits_{k=1}^n\bigcup\mathcal U_k:U\in\mathcal U_{n+1})\,\,\,\mbox{for}\,\,\, n\ge 1.
    $$
Then each set $A_n=\cup\mathcal A_n$ is functionally ambiguous in $X$ and the family  $(A_n:n\in\mathbb N)$  forms a partition of  $X$. Applying \cite[Proposition 8]{Karlova:TA:2015} we have that the formula
    $$
    g(x)=g_n(x),\,\,\,\mbox{if}\,\,\, x\in A_n\,\,\,\mbox{for some}\,\,\, n\in\mathbb N
    $$
   defines the map $g\in\Sigma_1^f(X,Y)$. Hence, $g\in {\rm B}_1(X,Y)$ by Theorem~\ref{th:LHB_general}. Moreover, for every $x\in X$ we have
    $$
    d_Y(f(x),g(x))\le\varepsilon,
    $$
since  $g$ is an extension of $g_n$, which is an extension of $f_n$.
\end{proof}

For a point $x$ of a topological space $X$ the system of all neighborhoods of $x$ is denoted by ${\mathcal U}_x$. For a map  $f:X\to Y$ between a topological space $X$ and a metric space $(Y,d_Y)$, and for a set $A\subseteq X$ we put
$$
\omega_f(A)=\mathop{\sup}\limits_{x,y\in A}d_Y(f(x),f(y)),\quad \omega_f(x)=\mathop{{\rm inf}}\limits_{U\in{\mathcal U}_x} \omega_f(U),
$$

We say that $f:X\to (Y,d_Y)$ is {\it weakly barely continuous}, if  for every $\varepsilon>0$ and for every non-empty closed set $F\subseteq X$ there exists a point $x\in F$ such that $\omega_{f|_F}(x)<\varepsilon$. Evidently, each barely continuous map is weakly barely continuous. The converse it not true, as a restriction of the Riemann function shows. Indeed, let $f(x)=\frac 1n$ if $x=r_n\in\mathbb Q$ for some $n\in\mathbb N$ and $f(x)=0$ if $x\in\mathbb R\setminus\mathbb Q$. Then $f|_{\mathbb Q}$ is weakly barely continuous, but is not a barely continuous function.

\begin{thm}\label{thm:barely_is_Baire1}
  Let $X$ be a perfect  paracompact space, $(Y,d_Y)$ be a metric contractible locally path-connected space and  $f:X\to Y$ be a weakly barely continuous map. Then $f\in{\rm B}_1(X,Y)$.
\end{thm}

\begin{proof}
 Fix $\varepsilon>0$ and show that $f$ can be locally $\varepsilon$-approximated by Baire-one maps. Let $\mathcal G$ be a collection of all open subsets of $X$ such that $f$ is $\varepsilon$-approximated by Baire-one maps on each set from $\mathcal G$. We put $G=\cup\mathcal G$ and notice that $G$ is open and non-empty, since it contains all continuity points of  $f$. Moreover, since  $G$ is $F_\sigma$ in  $X$, $G$ is a paracompact functionally open subset of  $X$. Applying Lemma~\ref{lem:local_uniform_approach} we obtain the existence of a map $h\in{\rm B}_1(G,Y)$ such that
  $$
  d_Y(h(x),f(x))\le\varepsilon
  $$
 for all $x\in G$.

 We put $F=X\setminus G$ and prove that $F=\emptyset$. Assume the contrary. Since $f$ is weakly barely continuous, we may choose a point  $x_0\in F$ and an open neighborhood $U_0$ of $x_0$ such that
 $$
  d_Y(f(x),f(x_0))<\varepsilon
  $$
 for all $x\in U_0\cap F$. Let $y_0=f(x_0)$. By Lemma~\ref{lem:ext_from_cozero} the map $g:X\to Y$, defined as follows
  $$
  g(x)=\left\{\begin{array}{ll}
                h(x), & x\in G, \\
                y_0, & x\in F,
              \end{array}
  \right.
  $$
 is of the first Baire class. Notice that
  $$
  d_Y(f(x),g(x))\le\varepsilon
  $$
for all $x\in U_0$. Therefore, $U_0\subseteq G$. In particular,  $x_0\in G$, which implies a contradiction. Hence, $G=X$.

Lemma~\ref{lem:local_uniform_approach} implies that there exists a sequence  $(f_n)_{n=1}^\infty$ of Baire-one maps wich that
  $$
  d_Y(f(x),f_n(x))\le\frac 1n.
  $$
Since a uniform limit of a sequence of Baire-one maps with  a metrizable connected and locally path-connected range space remains a Baire-one map~\cite[Theorem 2]{KarlovaMykhaylyuk:UMZh:2006},  $f\in {\rm B}_1(X,Y)$.
\end{proof}

\begin{thm}\label{th:barely_cont_is_sigma_discrete}
  Let $X$ be a perfect  paracompact space, $Y$ be a metrizable space and $f:X\to Y$ be a weakly barely continuous map. Then $f\in\Sigma_1(X,Y)$.
\end{thm}

\begin{proof}
 Let $Z=\ell_\infty(Y)$, $\varphi:Y\to Z$ is a homeomorphic embedding and $h=\varphi\circ f$. By Theorem~\ref{thm:barely_is_Baire1}, $h\in{\rm B}_1(X,Z)$. Then it follows from Theorem~\ref{th:LHB_general} that $h$ together with $f$ have a $\sigma$-discrete base $\mathcal B$ which consists of ambiguous sets in  $X$.
\end{proof}

\begin{rem}
Notice that if $X$ is not perfect, Theorem~\ref{th:barely_cont_is_sigma_discrete} is not valid. Indeed, let $D$ be an uncountable discrete space and $X=\alpha D$ be its one-point compactification. Consider a bijection $f:X\to D$. It is easy to see that $f$ is barely continuous. Assume that $f$ has a base $\mathcal B=\bigcup\limits_{n}\mathcal B_n$ in  $X$, where every family $\mathcal B_n$ is discrete in $X$. The compactness of $X$ implies that every family  $\mathcal B_n$ is finite. On the other hand, all singletons are elements of  $\mathcal B$, which implies a contradiction.
\end{rem}

\begin{thm}\label{th:general_sigma_H_barely}
 Let $X$ be a hereditarily Baire metrizable space and $Y$ be a metrizable space. For a map $f:X\to Y$ the following conditions are equivalent:
  \begin{enumerate}
    \item\label{it:1} $f\in\Sigma_1(X,Y)$,

    \item\label{it:2} $f\in {\rm H}_1(X,Y)$,

    \item\label{it:3} $f$ is barely continuous,

    \item\label{it:4} for any $\varepsilon>0$ there exists a function $\delta\in{\rm B}_1(X,\mathbb R^+)$ such that (\ref{gath:intro3}) for all $x,y\in X$, where $d_X$ is a metric on $X$ which generates its topology;

   \item\label{it:5} for any $\varepsilon>0$ there exists a function $\delta:X\to \mathbb R^+$ such that   (\ref{gath:intro3}) holds  for all $x,y\in X$, where $d_X$ is a metric on $X$ which generates its topology.
  \end{enumerate}
  \end{thm}

\begin{proof} The implication {\bf (\ref{it:1})$\Rightarrow$(\ref{it:2})} follows from~\cite[Proposition 2]{Hansell:1974}.
The implication {\bf (\ref{it:2})$\Rightarrow$(\ref{it:3})} was proved by G.~Koumoullis \cite[Theorem 4.12]{Kum}.
Theorem~\ref{th:barely_cont_is_sigma_discrete} implies  {\bf (\ref{it:3})$\Rightarrow$(\ref{it:1})}.
Lemma~\ref{lem:sigma_implies_ed} implies {\bf (\ref{it:1})$\Rightarrow$(\ref{it:4})}. Evidently, {\bf (\ref{it:4})$\Rightarrow$(\ref{it:5})}.
 Finally,  the implication {\bf (\ref{it:5})$\Rightarrow$(\ref{it:1})} is proved in Lemma~\ref{lem:ed_implies_sigma}.
\end{proof}

\begin{rem}
 The implication (\ref{it:3})$\Rightarrow$(\ref{it:2}) for a perfect  paracompact space $X$ and a perfect space $Y$ was established in~\cite[Theorem~6]{KS:MatStud:2012}. The implication (\ref{it:5})$\Rightarrow$(\ref{it:3}) was also proved in~\cite[Theorem~5]{KS:MatStud:2012}.
  \end{rem}

\section{A characterization of right compositors}

The class of all right ${\rm B}_\alpha$-compositors between $X$ and $Y$ we will denote by ${\rm RB}_\alpha(X,Y)$.

\begin{prop}\label{prop:stable_in_comp}
  Let $X$, $Y$ be topological spaces and $\alpha\in[1,\omega_1)$. Then
  \begin{gather*}
    \overline{\bigcup\limits_{\beta<\alpha}{\rm RB}_{\beta}(X,Y)}^{\,\,{\rm st}}\subseteq {\rm RB}_\alpha(X,Y).
  \end{gather*}
\end{prop}

\begin{proof}
 Let $(f_n)_{n=1}^\infty$ be a sequence of right ${\rm B}_{\beta_n}$-compositors which is  convergent pointwisely to $f:X\to Y$ and $\beta_n<\alpha$ for every $n\in\mathbb N$. Consider a topological space $Z$ and a map $g\in{\rm B}_\alpha(Y,Z)$. Choose a sequence $(g_n)_{n=1}^\infty$ of maps $g_n\in{\rm B}_{\gamma_n}(Y,Z)$ which converges pointwisely to $g$ and $\gamma_n<\alpha$ for every $n\in\mathbb N$. Denote $\delta_n=\max\{\beta_n,\gamma_n\}$ and observe that each map $f_n$ is the right  ${\rm B}_{\delta_n}$-compositor. Then $h_n=g_n\circ f_n\in {\rm B}_{\delta_n}(X,Z)$. Moreover, if $x\in X$, then $f_n(x)=f(x)$ for all $n\ge N$, which implies that $\lim\limits_{n\to\infty}h_n(x)=\lim\limits_{n\to\infty}g_n(f(x))=g(f(x))$. Hence, $f\in {\rm RB}_\alpha(X,Y)$.
\end{proof}

The following result was proved in~\cite[3.2 (4), p.160]{Hansell:1971} for metrizable spaces $X,Y,Z$.

\begin{lem}\label{lem:compose_sigma_discr}
 Let $X$, $Y$, $Z$ be topological spaces, $f\in\Sigma^f(X,Y)$, $g\in\Sigma^f(Y,Z)$ and $h=g\circ f$. Then $h\in\Sigma^f(X,Z)$.
\end{lem}

\begin{proof}
 Let $\mathcal A=\bigcup\limits_n\mathcal A_n$ be a $\sigma$-sdf base for $f$, $\mathcal B=\bigcup\limits_n\mathcal B_n$ be a $\sigma$-sfd base for  $g$ and each of the families $\mathcal A_n$ and $\mathcal B_n$ is strongly functionally discrete in $X$ and $Y$, respectively. For every $n$ we choose a discrete family  $\mathcal U_n=(U_B:B\in\mathcal B_n)$ of functionally open sets in $Y$ such that  $B\subseteq U_B$ for every  $B\in\mathcal B_n$. For all  $n,m\in\mathbb N$ we put
  $$
  \mathcal C_{nm}=\{A\cap f^{-1}(B): A\in\mathcal A_n, B\in\mathcal B_m\,\,\,\mbox{and}\,\,\, f(A)\subseteq U_B\}.
  $$
  Then the family $\mathcal C_{nm}$ is strongly functionally discrete, since each $\mathcal A_n$ is strongly functionally discrete and for every $A\in\mathcal A_n$ there exists at most one $B\in\mathcal B_m$ with $f(A)\subseteq U_B$.

  It remains to show that the family
  $$
  \mathcal C=\bigcup\limits_n\bigcup\limits_m\mathcal C_{nm}
  $$
is a base for $h$. Let $W$ be an open set in $Z$. Since $\mathcal B$ is a base for  $g$, $g^{-1}(W)=\bigcup\mathcal B_W$, where $\mathcal B_W\subseteq\mathcal B$. Then
\begin{gather}\label{gath:formula1}
  h^{-1}(W)=\bigcup \{f^{-1}(B):B\in\mathcal B_W\}.
   \end{gather}
For each $B\in\mathcal B_W$ we take a family $\mathcal A_B\subseteq\mathcal A$ such that $f^{-1}(U_B)=\cup\mathcal A_B$.
Then
$$
f^{-1}(B)=\bigcup\{A\cap f^{-1}(B):A\in\mathcal A_B, f(A)\subseteq U_B\}
$$
for every $B\in\mathcal B_W$. Taking into account  (\ref{gath:formula1}), we obtain that
$$
h^{-1}(W)=\bigcup\{A\cap f^{-1}(B): B\in\mathcal B_W, A\in\mathcal A_B, f(A)\subseteq U_B\}.
$$
Therefore,  $\mathcal C$ is a base for $h$.
\end{proof}

Let $\alpha\in[0,\omega_1)$. A map $f:X\to Y$ is called {\it $\mathcal M_\alpha$-measurable}, if the preimage $f^{-1}(B)$ of any set $B\in\mathcal M_{\alpha}(Y)$ belongs to the class $\mathcal M_\alpha(X)$.

\begin{rem}\label{rem:ambiguous}
  Evidently, a map is $\mathcal M_\alpha$-measurable if the preimage of any set from $\mathcal A_\alpha(Y)$ belongs to $\mathcal A_\alpha(X)$. Since each set of the $\alpha$'th functionally additive class is a union of a sequence of functionally ambiguous sets of the class $\alpha$, a map is $\mathcal M_\alpha$-measurable if and only if the preimage of any functionally ambiguous set of the class $\alpha$ in $Y$ is functionally ambiguous of the class $\alpha$ in $X$.
\end{rem}

\begin{thm}\label{th:B_alpha_M_alpha}
  Let $X$ be a topological space, $Y$ be a metrizable space,  $\alpha\in[1,\omega_1)$ and $\mathcal Z$ be the class of all metrizable connected and locally path-connected spaces. Let $\beta=\alpha$ if $\alpha<\omega_0$ and $\beta=\alpha+1$ if $\alpha\ge \omega_0$. For a map  $f:X\to Y$ the following conditions are equivalent:
  \begin{enumerate}
    \item\label{th:B_alpha_M_alpha:it:1}  $f$ is the right  ${\rm B}_\alpha$-compositor for the class $\mathcal Z$;

    \item\label{th:B_alpha_M_alpha:it:2}  $f$ is $\mathcal M_\beta$-measurable and $\sigma$-strongly functionally discrete.
  \end{enumerate}
\end{thm}

\begin{proof} {\bf (\ref{th:B_alpha_M_alpha:it:1})$\Rightarrow$(\ref{th:B_alpha_M_alpha:it:2}).} Keeping in mind Remark~\ref{rem:ambiguous}, we consider a functionally ambiguous set $B$ of the class  $\beta$ in  $Y$ and its characteristic function  $g:Y\to \mathbb R$. Notice that $g\in {\rm H}_\beta(Y,\mathbb R)$, which implies that  $g\in {\rm B}_\alpha(Y,\mathbb R)$ by \cite[p.~393]{Kuratowski:Top:1}. Then $g\circ f\in {\rm B}_\alpha(X,\mathbb R)=\Sigma_\beta^f(X,\mathbb R)$. Since
$$
f^{-1}(B)=(g\circ f)^{-1}((0,2))=(g\circ f)^{-1}(\{1\}),
$$
the set  $f^{-1}(B)$ is functionally ambiguous of the class  $\beta$ in  $X$. Therefore, $f$ is  $\mathcal M_\beta$-measurable.

Now we consider the space $Z=\ell_\infty(Y)\in\mathcal Z$ and a homeomorphic embedding  $\varphi:Y\to Z$. Condition (\ref{th:B_alpha_M_alpha:it:1}) implies that the composition   $\tilde f=\varphi\circ f:X\to Z$ belongs to the class ${\rm B}_\alpha(X,Z)$. By \cite[Corollary 19]{Karlova:TA:2015}, $\tilde{f}\in\Sigma^f(X,Z)$. It is easy to see that a $\sigma$-sfd base $\mathcal B$ for $\tilde f$ is also a base for $f$. Hence, $f\in\Sigma^f(X,Y)$.

{\bf (\ref{th:B_alpha_M_alpha:it:2})$\Rightarrow$(\ref{th:B_alpha_M_alpha:it:1}).} Let $Z\in\mathcal Z$ and $g\in{\rm B}_\alpha(Y,Z)$. Denote $h=g\circ f$.
We show that  $h\in{\rm K}_\beta(X,Z)$. Let  $F$ be a closed set in $Z$.  Then  $g^{-1}(F)\in\mathcal M_\beta(Y)$, which implies that $h^{-1}(F)=f^{-1}(g^{-1}(F))\in\mathcal M_\beta(X)$, because $f$ is $\mathcal M_\beta$-measurable.

Notice that  $g\in \Sigma^f(Y,Z)$.  It follows from Lemma~\ref{lem:compose_sigma_discr} that $h$ is a $\sigma$-sfd map.

Hence, $h\in {\Sigma}_\beta^f(X,Z)$ according to~\cite[Theorem 6]{Karlova:TA:2015}. Theorem~\ref{th:LHB_general} implies that $h\in {\rm B}_\alpha(X,Z)$.
\end{proof}

\begin{rem}
  If $\alpha=0$, then conditions (\ref{th:B_alpha_M_alpha:it:1}) and (\ref{th:B_alpha_M_alpha:it:2}) are equivalent to the continuity of $f$.
\end{rem}

 A map  $f:X\to Y$ between topological spaces is {\it  functionally piecewise continuous} if there exists a cover $(F_n:n\in\mathbb N)$ of  $X$ by functionally closed sets such that each restriction $f|_{F_n}$ is continuous.

 For a topological space $X$ and a pair $(Y,B)$ of a topological space $Y$ and its subspace $B$ we shall write $(Y,B)\in {\rm AE}^f(X)$ if each continuous map $f:F\to B$ defined on a functionally closed set $F\subseteq X$ can be extended to a continuous map $g:X\to Y$. We define a space $Y$ to belong to the class $\sigma {\rm AE}^f(X)$ if there exists a countable cover $(Y_n:n\in\mathbb N)$ of $Y$ by closed $G_\delta$-sets such that $(Y,Y_n)\in {\rm AE}^f(X)$ for all $n\in\mathbb N$.

  \begin{thm}\label{tururu}
Let  $X$ be a topological space, $Y$ be a metrizable space. For a map $f:X\to Y$ the following conditions are equivalent:
    \begin{enumerate}
       \item\label{itt:1} $f\in {\rm B}_1^{\rm st}(X,Y)$;

      \item\label{itt:2} $f$ is functionally piecewise continuous;

      \item\label{itt:3} for any function $\varepsilon:Y\to \mathbb R^+$ there exists a neighborhood-map $U:X\to\tau(X)$ such that
    \begin{gather}\label{gath:char:comp}
      x,y\in U(x)\cap U(y)\,\,\Longrightarrow \,\, d_Y(f(x),f(y))<\min\{\varepsilon(f(x)),\varepsilon(f(y))\}
    \end{gather}
for all $x,y\in X$;

      \item\label{itt:4}  for any function $\varepsilon\in{\rm B}_1(Y,\mathbb R^+)$ here exists a neighborhood-map $U:X\to\tau(X)$ such that  (\ref{gath:char:comp}) holds for all $x,y\in X$;

     \item\label{itt:5}  $f$ is the right ${\rm B}_1$-compositor for the class $\mathcal Z$ of all metrizable connected and locally path-connected spaces;

      \item\label{itt:6} $f$ is $\mathcal M_1$-measurable and $\sigma$-strongly functionally discrete.
    \end{enumerate}
Then (\ref{itt:1})$\Rightarrow$(\ref{itt:2})$\Rightarrow$(\ref{itt:3})$\Rightarrow$(\ref{itt:4}) and (\ref{itt:5})$\Leftrightarrow$(\ref{itt:6}).

If  $X$  is a metrizable space, then  (\ref{itt:4})$\Rightarrow$(\ref{itt:5}).

If $X$ is a hereditarily Baire perfect paracompact space Preiss-Simon space, then (\ref{itt:6})$\Rightarrow$(\ref{itt:2}).

If $Y$ is a path-connected space and  $Y\in\sigma {\rm AE}^f(X)$, then (\ref{itt:2})$\Rightarrow$(\ref{itt:1}).
\end{thm}

\begin{proof} {\bf (\ref{itt:1})$\Rightarrow$(\ref{itt:2}).} Denote $\Delta=\{(y,y):y\in Y\}$ and let $(f_n)_{n=1}^\infty$ be a sequence of continuous maps $f_n:X\to Y$ which converges stably to $f$. For  $k,n\in\mathbb N$ we put
  \begin{gather*}
X_{k,n}=\{x\in X:f_k(x)=f_n(x)\}\quad\mbox{and}\quad  X_n=\bigcap\limits_{k=n}^\infty X_{k,n}.
  \end{gather*}
Clearly, $X_n\subseteq X_{n+1}$, $X=\bigcup\limits_{n=1}^\infty X_n$ and $f_n|_{X_n}=f|_{X_n}$ for every $n\in\mathbb N$.

For all $x\in X$ and $k,n\in\mathbb N$ we consider a continuous map $h_{k,n}:X\to Y^2$,  $h_{k,n}(x)=(f_k(x),f_n(x))$. Since
$X_{k,n}=h_{k,n}^{-1}(\Delta)$ and the set $\Delta$ is functionally closed in $Y^2$, $X_{k,n}$ is functionally closed in $X$.

{\bf (\ref{itt:2})$\Rightarrow$(\ref{itt:3}).} Let $(F_n)_{n=1}^\infty$ be a sequence of functionally closed sets which covers   $X$ and the restriction   $f|_{F_n}$ is continuous for every $n$. We put $A_1=F_1$ and $A_n=F_n\setminus (F_1\cup\dots\cup F_{n-1})$ for $n>1$. Fix  $\varepsilon:Y\to\mathbb R^+$. Lemma~\ref{lem:neighborhood_map} allows to take a neighborhood-map $V:X\to\tau (X)$ satisfying condition (\ref{gath:neighborhood_map_lem}) for some $\sigma$-sfd cover $\mathcal A\prec (A_n:n\in\mathbb N)$ of the space $X$ by functionally closed sets. For every   $n\in\mathbb N$ we use the continuity of  the restriction $f_n|_{A_n}$ and choose a neighborhood-map $U_n:A_n\to\tau(A_n)$ such that for all  $x,y\in X$ the inclusion  $x,y\in U_n(x)\cap U_n(y)$ implies the inequality  $d_Y(f_n(x),f_n(y))<\min\{\varepsilon(f_n(x)),\varepsilon(f_n(y))\}$. For every $n$ we take a map $V_n:A_n\to\tau(X)$ such that $V_n(x)\cap A_n=U_n(x)$ for all $x\in A_n$. Now for all $x\in X$ we put
$$
U(x)=V(x)\cap V_n(x),
$$
if $x\in A_n$ for some $n$. It is easy see that the neighborhood-map $U:X\to\tau(X)$ satisfies the required properties.

The implication {\bf (\ref{itt:3})$\Rightarrow$(\ref{itt:4})} is evident.

Assume that  $X$ is a metrizable space and prove {\bf (\ref{itt:4})$\Rightarrow$(\ref{itt:5}).} Let  $Z\in\mathcal Z$, $d_X$ and $d_Z$ be metrics on  $X$ and  $Z$, which generate topological structures of these spaces and let  $g\in {\rm B}_1(Y,Z)$. We shoe that the composition $h=g\circ f$ belongs to  $\Sigma_1^f(X,Z)$. Fix $\varepsilon>0$. Applying Theorem~\ref{th:LHB_general},  we get $g\in\Sigma_1^f(Y,Z)$. Then Lemma~\ref{lem:sigma_implies_ed}~(b) implies that there exists a function  $\gamma\in{\rm B}_1(Y,\mathbb R^+)$ such that for all  $x,y\in Y$ we have
$$
d_Y(x,y)<\min\{\gamma(x),\gamma(y)\}\,\,\,\Longrightarrow\,\,\, d_Z(g(x),g(y))<\varepsilon.
$$
According to condition (\ref{itt:4}) there exists a neighborhood-map $U:X\to\tau(X)$ such that $x,y\in X$
    \begin{gather*}
      x,y\in U(x)\cap U(y)\,\,\Longrightarrow \,\, d_Y(f(x),f(y))<\min\{\gamma(f(x)),\gamma(f(y))\}.
    \end{gather*}
For all  $x\in X$ we put
$$
\delta(x)=d_X(x,X\setminus U(x))
$$
and obtain a continuous function $\delta:X\to\mathbb R^+$ which satisfies the condition of Lemma~\ref{lem:ed_implies_sigma} for the function $h:X\to Z$. Hence, $h\in\Sigma^f_1(X,Z)$. Again by Theorem~\ref{th:LHB_general} we get  $h\in{\rm B}_1(X,Z)$.

The equivalence {\bf (\ref{itt:5})$\Leftrightarrow$(\ref{itt:6})} was proved in Theorem~\ref{th:B_alpha_M_alpha}.

The implication  {\bf (\ref{itt:6})$\Rightarrow$(\ref{itt:2})} follows from~\cite[Theorem 8.1]{BaBo}.

The implication {\bf (\ref{itt:2})$\Rightarrow$(\ref{itt:1})} can be proved completely similarly to the proof of Theorem~ 6.3 from  \cite{BaBo}, which shows this implication for a normal space $X$ and a path-connected $Y\in{\sigma}{\rm AE}^f(X)$.
\end{proof}

\begin{rem}
 According to Proposition~\ref{prop:stable_in_comp}, a stable limit of right ${\rm B}_1$-compositors for the class of metrizable connected and locally path-connected spaces is the right ${\rm B}_2$-compositor. Then each function $f\in{\rm B}_2^{\rm st}(\mathbb R,\mathbb R)$ is the right ${\rm B}_2$-compositor. We observe that the class of all right ${\rm B}_2$-compositors is strictly wider than the class  ${\rm B}_2^{\rm st}(\mathbb R,\mathbb R)$. Indeed, consider the increasing function   $f:\mathbb R\to\mathbb R$,
  $$
  f(x)=\sum\limits_{r_n\le x}\frac{1}{2^n},
  $$
where $\mathbb Q=\{r_n:n\in\mathbb N\}$.Since  $f\in{\rm B}_1(\mathbb R,\mathbb R)$,  $f$ is $\mathcal M_2$-measurable. By Theorem~\ref{th:B_alpha_M_alpha},   $f$ is the right  ${\rm B}_2$-compositor. But the discontinuity points set of  $f$ is $\mathbb Q$. Therefore,   $f\not\in \bigcup\limits_{\alpha<\omega_1}{\rm B}_\alpha^{\rm st}(\mathbb R,\mathbb R)$.
\end{rem}

\section{Left compositors}
 Let
 \begin{itemize}
  \item  ${\rm H}_\alpha^0(X,Y)$ be a set of all maps from ${\rm H}_\alpha(X,Y)$ with a finite range;

\item ${\rm H}_{\alpha}^{\mathcal M}(X,Y)$  be a set of all maps from ${\rm H}_\alpha(X,Y)$ with a metrizable separable range.
\end{itemize}

The next result follows from~\cite[pp. 389--391]{Kuratowski:Top:1}.

\begin{thm}\label{th:cor:Kuratowski}
Let $X$ be a perfectly normal space, $Y$ be a topological space and $\alpha\in [1,\omega_1)$. Then
\begin{enumerate}
\item[(i)] ${\rm H}_{\alpha+1}^{\mathcal M}(X,Y)\subseteq\overline{{\rm H}_{\alpha}^0(X,Y)}^{\,\,{\rm p}}$, if $\alpha$ is isolated;

\item[(ii)] ${\rm H}_{\alpha+1}^{\mathcal M}(X,Y)\subseteq\overline{{\rm H}_{<\alpha}^0(X,Y)}^{\,\,{\rm p}}$, if  $\alpha$ is limit.
\end{enumerate}
\end{thm}

We call a topological space  $X$  {\it densely connected}, if for any open non-empty sets $U_1,\dots,U_n$ in $X$ there exists a continuous map $\gamma:[1,n]\to X$ such that $\gamma(i)\in U_i$ for every $i\in\{1,\dots,n\}$.

Notice that each space with dense path-connected subspace is densely connected.

  \begin{lem}\label{lemma:almostBaire}
    Let $X$ be a perfectly normal space, $Y$ be a densely connected first countable $T_1$-space and $\alpha\in[1,\omega_1)$. Then
    \begin{enumerate}
      \item\label{it:lemma:almostBaire:1} ${\rm H}_\alpha^0(X,Y)\subseteq{\rm B}_\alpha(X,Y)$, if $\alpha<\omega_0$;

      \item\label{it:lemma:almostBaire:2} ${\rm H}_{\alpha+1}^0(X,Y)\subseteq{\rm B}_\alpha(X,Y)$, if $\alpha\ge\omega_0$.
    \end{enumerate}
  \end{lem}

  \begin{proof} {\bf (\ref{it:lemma:almostBaire:1}).} Assume $\alpha<\omega_0$, $f\in {\rm H}_\alpha^0(X,Y)$, $f(X)=\{y_1,\dots,y_n\}$, where $y_i\ne y_j$ for  $i\ne j$, and $A_i=f^{-1}(y_i)$, $i=1,\dots,n$. Then $X=\bigcup\limits_{i=1}^n A_i$ and each set  $A_i$ is ambiguous of the class $\alpha$. For every  $i=1,\dots,n$ we take an increasing sequence  $(A_{i,k})_{k=1}^\infty$ of sets of multiplicative classes $<\alpha$ and a decreasing sequence  $(V_{i,k})_{k=1}^\infty$ of open neighborhoods of $y_i$ such that
    \begin{gather*}
      A_i=\bigcup\limits_{k=1}^\infty A_{i,k}\quad\mbox{and}\quad  \{y_i\}=\bigcap\limits_{k=1}^\infty V_{i,k}.
    \end{gather*}

Since for every $k\in\mathbb N$ the family $(A_{i,k}:i=1,\dots,n)$ is disjoint, it follows from \cite[Lemma 2.1]{Karlova:AGT_2012} that there exists a function $g_k\in {\rm B}_{<\alpha}(X,[1,n])$ such that
    $$
    A_{i,k}=g_k^{-1}(i)
    $$
for all $i=1,\dots,n$. Moreover, since $Y$ is densely connected, for every $k\in\mathbb N$ we choose a continuous map $\gamma_k:[1,n]\to Y$ such that
    $$
    \gamma_k(i)\in V_{i,k}
    $$
  for  $i=1,\dots,n$. We put
    $$
    f_k=\gamma_k\circ g_k
    $$
    for every $k\in\mathbb N$
and obtain a sequence $(f_k)_{k=1}^\infty$ of maps $f_k\in{\rm B}_{<\alpha}(X,Y)$ which is pointwise convergent to $f$ on $X$. Hence, $f\in {\rm B}_\alpha(X,Y)$.

    {\bf (\ref{it:lemma:almostBaire:2}).} We argue by the induction on  $\alpha$. Let $\alpha=\omega_0$ and $f\in {\rm H}_{\alpha+1}^0(X,Y)$. By Theorem~\ref{th:cor:Kuratowski} there exists a sequence of maps  $f_n\in{\rm H}_{n}^0(X,Y)$ which is pointwise convergent to $f$ on $X$. It follows from the previous arguments that  $f_n\in {\rm B}_n(X,Y)$. Therefore,  $f\in {\rm B}_{\omega_0}^0(X,Y)$.

    Now we suppose that (\ref{it:lemma:almostBaire:2}) is true for all $\beta\in[\omega_0,\alpha)$, where $\alpha<\omega_1$, and consider a map  $f\in {\rm H}_{\alpha+1}^0(X,Y)$.

   Let $\alpha$ be a limit ordinal. Applying Theorem~\ref{th:cor:Kuratowski} we choose a sequence of maps $(f_n)_{n=1}^\infty$ which converges pointwise to  $f$ on  $X$ and $f_n\in {\rm H}_{<\alpha}^0(X,Y)$. By the inductive assumption we have $f_n\in {\rm B}_{<\alpha}(X,Y)$ which implies that  $f\in {\rm B}_{\alpha}(X,Y)$.

If $\alpha$ is isolated, then $\alpha=\lambda+m$,  where $\lambda$ is limit and  $m\in\mathbb N$.  By the induction on $m$ we obtain that  ${\rm H}_{\lambda+m+1}^0(X,Y)\subseteq {\rm B}_{\lambda+m}(X,Y)$.
  \end{proof}

Theorem~\ref{th:cor:Kuratowski}  and Lemma~\ref{lemma:almostBaire} imply
  \begin{thm}\label{th:high_classes_sep}
Let $X$ be a perfectly normal space, $Y$ be a first countable densely connected $T_1$-space and $\alpha\in[1,\omega_1)$. Then
    \begin{enumerate}
      \item ${\rm H}_{\alpha+1}^{\mathcal M}(X,Y)\subseteq{\rm B}_{\alpha+1}(X,Y)$, if $\alpha<\omega_0$;

      \item ${\rm H}_{\alpha+1}^{\mathcal M}(X,Y)\subseteq{\rm B}_\alpha(X,Y)$, if $\alpha\ge\omega_0$.
    \end{enumerate}
  \end{thm}

  \begin{thm}\label{tarara}
Let $X$ be a $T_1$-space, $Y$ be a perfectly normal space, $f:X\to Y$ be a map and $\alpha\in[1,\omega_1)$. If
\begin{itemize}
  \item[(a)] $X$ is a connected and locally path-connected metrizable space and $\alpha=1$, or

  \item[(b)] $X$ is a first countable densely connected space and  $\alpha>1$,
\end{itemize}
then the following conditions are equivalent:
   \begin{enumerate}
      \item\label{it:tarara:1} $f$ is continuous;

      \item\label{it:tarara:2} $f$ is the left $B_\alpha$-compositor.
      \end{enumerate}
  \end{thm}

  \begin{proof} Since the implication (\ref{it:tarara:1})$\Rightarrow$(\ref{it:tarara:2}) is obvious, we prove (\ref{it:tarara:2})$\Rightarrow$(\ref{it:tarara:1}). Assume that $f$ is discontinuous at $x_0\in X$  and choose a sequence $(x_n)_{n=1}^\infty$ of points from  $X$ and a neighborhood $V$ of $y_0=f(x_0)$ in $Y$ such that $\lim\limits_{n\to\infty} x_n=x_0$ and $y_n=f(x_n)\in Y\setminus V$ for every $n\in\mathbb N$.

We take a set $A\subseteq \mathbb R$ such that $A\in{\mathcal A}_\beta\setminus {\mathcal M}_\beta$, where $\beta=\alpha$ if  $\alpha<\omega_0$,  and $\beta=\alpha+1$ if $\alpha\ge\omega_0$. Let $(A_n)_{n=1}^\infty$ be a sequence of sets such that $A=\bigcup\limits_{n=1}^\infty A_n$ and $A_n\in\mathcal M_{<\beta}$. We set
  \begin{gather*}
  B_1=A_1\quad\mbox{and}\quad B_n=A_n\setminus (A_1\cup\dots\cup A_{n-1})\,\,\,\mbox{for}\,\,\, n\ge 2.
  \end{gather*}
Then $(B_n)_{n=1}^\infty$  is a disjoint sequence of ambiguous sets of the class  $\beta$. Moreover,  $A=\bigcup\limits_{n=1}^\infty B_n$.

We define a map $g:\mathbb R\to X$ as follows:
\begin{gather}\label{gath:riemann}
    g(t)=\left\{\begin{array}{ll}
                 x_n,  & t\in B_n\,\,\,\mbox{for some}\,\,\, n\in\mathbb N, \\
                  x_0, & t\in\mathbb R\setminus A
                \end{array}
    \right.
\end{gather}
and show that $g\in {\rm B}_\alpha(\mathbb R,X)$. Since  the sequence $(x_n)_{n=1}^\infty$ converges to $x_0$,  $g\in {\rm H}_\beta(\mathbb R,X)$. In the case  (a) it follows from  \cite[Theorem 1]{Fos}  that $g\in {\rm B}_1(\mathbb R,X)$. In the case $(b)$ we observe that the space  $X_0=\{x_n:n=0,1,\dots\}$ is metrizable and separable. Then  we have that $g\in {\rm B}_\alpha(\mathbb R,X)$ by Theorem~\ref{th:high_classes_sep}.

According to the assumption, the composition  $h=f\circ g:\mathbb R\to Y$ belongs to the class ${\rm B}_\alpha(\mathbb R,Y)$. On the other hand,
$$
h(t)=\left\{\begin{array}{ll}
                 y_n, & t\in B_n\,\,\,\mbox{for some}\,\,\, n\in\mathbb N, \\
                 y_0, &  t\in\mathbb R\setminus A.
               \end{array}
\right.
$$
Then $h^{-1}(V)=\mathbb R\setminus A$. Hence,  $h\in {\rm H}_\beta(\mathbb R,Y)$ by Theorem 1 from  \cite[p. 386]{Kuratowski:Top:1}. This implies a contradiction, since $\mathbb R\setminus A\not\in\mathcal A_\beta$.
\end{proof}

\end{document}